\documentclass[12pt,a4paper]{amsart}
\usepackage{amssymb,amsmath,amsfonts,latexsym,amsthm}
\usepackage{mathtools}
\usepackage{tikz}
\usetikzlibrary{graphs}
\usetikzlibrary {arrows.meta,bending,positioning}
\usepackage{tikz-cd}
\usepackage{calc}
\usepackage{color}
\newcommand{\rosso}[1]{{\color{red}{#1}}}

\usepackage[latin2]{inputenc}
\usepackage[theoremfont]{newtxtext}
\usepackage{newtxmath}
\usepackage{graphicx}
\allowdisplaybreaks
\usepackage[matrix,arrow]{xy}
\usepackage{enumitem}
\def\Im{\mathop{\rm Im}\nolimits}

\def\Zb{{\mathbb Z}}

\def\0D{\Delta^{(0)}}
\def\1D{\Delta^{(1)}}

\newtheorem{theorem}{Theorem}[section]

\newtheorem{proposition}[theorem]{Proposition}
\newtheorem{lemma}[theorem]{Lemma}

\newtheorem{definition}[theorem]{Definition}

\def\build#1_#2^#3{\mathrel{
\mathop{\kern 0pt#1}\limits_{#2}^{#3}}}

\numberwithin{equation}{section}

\def\part{\partial}

\def\text{\hbox}

\def\build#1_#2^#3{\mathrel{
\mathop{\kern 0pt#1}\limits_{#2}^{#3}}}

\newcommand{\set}[1]{\left\{#1\right\}}

\newcommand{\brac}[1]{\ensuremath{\left( {#1} \right)}}
\newcommand{\sbrac}[1]{\ensuremath{\left[ {#1} \right]}}

\numberwithin{equation}{section}

\newcommand{\comment}[1]{\relax}

\title{\vspace*{-15mm}  Unital embeddings of Cuntz algebras\\ from path homomorphisms of graphs}
\author[P. M.~Hajac]{Piotr M.~Hajac}
\address[P. M.~Hajac]{Instytut Matematyczny, Polska Akademia Nauk, ul. \'Sniadeckich 8, Warszawa, 00-656 Poland}
\email{pmh@impan.pl }

\author[Y. Liu]{Yang Liu}
\address[Y. Liu]{SISSA, via Bonomea 265, 34136 Trieste  }
\email{yliu@sissa.it}

\begin{document}

\begin{abstract}\vspace*{-3mm}
Cuntz algebras $\mathcal{O}_n$, $n>1$, are celebrated examples of a separable infinite simple C*-algebra with a number of 
fascinating properties. Their K-theory allows an embedding of $\mathcal O_m$ in $\mathcal O_n$ whenever $n-1$ divides $m-1$. In 
2009, Kawamura provided a simple and explicit formula for all such embeddings. His formulas can be easily deduced by viewing Cuntz 
algebras as graph C*-algebras. Our main result is that, using both the covariant and contravariant functoriality of assigning graph C*-
algebras to directed graphs, we can provide explicit polynomial formulas for all unital embeddings of Cuntz algebras into 
matrices over Cuntz algebras allowed by K-theory.
\end{abstract}
\maketitle
\vspace*{-10mm}\tableofcontents 

\section{Introduction}
\noindent
Let $m$ and $n$ be  positive integers bigger than~$1$. Since $K_0(\mathcal O_n)=\mathbb Z/(n-1)\mathbb Z$ with $[1]=1$,
one sees immediately that a unital $*$-homomorphism $\mathcal O_m\to \mathcal O_n$ can exist only if 
$m-1 = k (n-1)$ for some  integer~$k$. We call them unital embeddings because the Cuntz algebras are simple \cite{c-j77}, so these
$*$-homomorphisms are automatically injective.
In~\cite{k-k09},  Kawamura constructs all such  $*$-homomorphisms
by assigning to generators monomials of generators. 

Furthermore, it turns out that tensoring the target Cuntz algebra with a matrix algebra yields more general unital embeddings.
More precisely, this time K-theory implies that a unital $*$-homomorphism $\mathcal O_p\to M_k(\mathcal O_q)$ can
exist only if
\begin{align}
(p-1)k = (q-1)s \;\text{for some}\; s\in\mathbb{N}\setminus\{0\}.
\end{align}
Note that, in this more general setting, the number of standard generators of the source Cuntz algebra no longer has to match
or exceed the number of standard generators of the target  Cuntz algebra.

On the other hand, Cuntz algebras are graph C*-algebras \cite{r-i05}, 
and a wide class $*$-homomorphisms between graph C*-algebras
was recently unraveled in \cite{ht22, ht24} as coming functorialy from morphisms of graphs. 
The main goal of this  paper is, using the aforementioned functoriality results, to construct all unital embeddings
$\mathcal O_p\to M_k(\mathcal O_q)$ by assigning to generators polynomials of generators. On the way
we show that the Kawamura formula can be easily derived from the covariant functoriality of
graph C*-algebras.

\section{Preliminaries}

\subsection{Directed graphs and finite paths}

\begin{definition}
A graph (directed graph, quiver) $E:=(E^0,E^1,s,t)$ 
is a quadruple consisting of the set of vertices $E^0$, the
set of edges (arrows) $E^1$, 
and the source and target (range) maps $s_E,\,t_E\colon E^1\to E^0$ assigning to each
edge its source and target vertex respectively. A graph is called finite iff both $E^0$ and $E^1$ are finite.
\end{definition}

Let $v$ be a vertex in a graph $E$. It is called a {\em sink} iff $s^{-1}_E(v)=\emptyset$,  it is called a {\em source} iff~$t^{-1}_E
(v)=\emptyset$, and it is called {\em regular} iff it is not a sink and $|s_E^{-1}(v)|<\infty$. The subset of regular vertices of a graph 
$E$ is denoted by ${\rm reg}(E)$. A {\em finite path} $p$ in $E$ is a vertex or a finite collection $e_1\ldots e_n$ of edges satisfying 
\begin{equation}
t_E(e_1)=s_E(e_2),\quad t_E(e_2)=s_E(e_3),\quad \ldots,\quad t_E(e_{n-1})=s_E(e_n).
\end{equation} 
Vertices are considered as finite paths of length~$0$. Every edge is a path of 
length~$1$. The {\em length} of a finite path that is not a vertex is the size of the tuple. 
A {\em loop} is a positive-length path that begins and  ends at the same vertex.

We denote the set of all finite-length paths in $E$ by~$FP(E)$. 
The beginning $s_E(p)$ of $p$ is $s_E(e_1)$ and the end $t_E(p)$ of $p$ is $t_E(e_n)$. The 
beginning and the end of a vertex is the vertex itself. We thus extend the source and target maps to maps
\begin{equation}
s_{PE},\,t_{PE}\colon FP(E)\longrightarrow E^0.
\end{equation}
 Finally, note that there exists a partial order on the set of finite paths:
\begin{equation}\label{popaths}
\alpha\preceq\beta\iff\exists\;\gamma\in FP(E)\colon \beta=\alpha\gamma.
\end{equation}

\subsection{Graph C*-algebras}

\begin{definition}[cf.\ \cite{ck80}]\label{graphc*}
Let $E$ be a graph. The {\em graph C*-algebra} $C^*(E)$ of $E$ is 
the universal  C*-algebra
generated by mutually orthogonal projections $P_v$, $v\in E^0$, and partial isometries $S_e$, $e\in E^1$, with mutually 
orthogonal ranges, satisfying
\begin{enumerate}
\vspace*{-3mm}\item $S^*_eS_e=P_{t_E(e)}$ for all $e\in E^1$,
\item $P_v=\sum_{e\in s^{-1}_E(v)}S_eS_e^*$ for all $v\in {\rm reg}(E)$,
\item $S_eS_e^*\leq P_{s_E(e)}$ for all $e\in E^1$.
\end{enumerate}
\end{definition}
\noindent
In what follows, we will need the notaion $S_p:=S_{e_1}S_{e_2}\ldots S_{e_n}$,
 where $p=e_1e_2\ldots e_n$ is a finite path in~$E$, and $S_p=P_v$ when $p=v$ is a vertex. 

\section{Functoriality}
\noindent
In this section, we briefly recall the contravariant and the covariant functors 
constructed in \cite{ht22} and \cite{ht24} respectively.

\subsection{Katsura's contravariant functoriality}
The standard category of graphs consists of graphs as objects equipped with morphisms in the
following sense:
\begin{definition}
A {\em homomorphism} from a graph $E:=(E^0,E^1,s_E,t_E)$ to a graph\linebreak $F:=(F^0,F^1,s_F,t_F)$
is a pair of maps 
\begin{equation*}
(f^0:E^0\to F^0,f^1:E^1\to F^1)
\end{equation*} 
satisfying the conditions:
\begin{equation*} 
s_F\circ f^1=f^0\circ s_E\,,\qquad t_F\circ f^1=f^0\circ t_E\,.
\end{equation*}
\end{definition}

The domain category of the contravariant functor assigning graph C*-algebras to graphs is obtained by restricting 
the graph homomorphisms as follows:
\begin{definition}[\cite{k-t06,ht22}]
We say that a graph homomorphism $(f_0,f_1)\colon E\to F$ is \emph{admissible} if it satisfies the three conditions:
\begin{enumerate}
\vspace*{-3mm}\item
Both $f_0$ and $f_1$ are proper maps (i.e.\ finite to one).
\item
$
\forall\; x\in F^1\colon (f^1)^{-1}(x)\ni e\longmapsto t_E(e) \in (f^0)^{-1}(t_F(x))\;\text{is bijective}
$ (target-bijectivity condition).
\item
$(f^0)^{-1}(\mathrm{reg}(F))\subseteq \mathrm{reg}(E)$ (regularity condition).
\end{enumerate}
\end{definition}

\begin{theorem}[\cite{k-t06,ht22}]
Directed graphs and admissible graph homomorphisms form a category, and
we have a contravariant functor from this category to the category 
of  C*-algebras and $*$-homomorphisms that assigns graph C*-algebras to graphs and  $*$-homomorphisms
to  admissible graph homomorphisms as follows. If
$ g : F \to E $ is an admissible graph homomorphism, then the formula  
\begin{align*}
g^* ( S_p )  := \sum_{q\in {g}^{-1}(p)} S_q,\quad p \in E^0 \cup E^1,
\end{align*}
defines a $*$-homomorphism $ g^* : C^*(E) \to C^*(F) $.
\end{theorem}

\subsection{The covariant functoriality}
The domain of the covariant functor needs the concept of an admissible  path homomorphism of graphs.
\begin{definition}[\cite{ht24}]
Let $E$ and $F$ be graphs.
A {\em path homomorphism} from $E$ to $F$  is a map \mbox{$f\colon FP(E)\to FP(F)$} satisfying:
\begin{enumerate}
\vspace*{-3mm}\item
$f(E^0)\subseteq F^0$,
\item
$s_{PF}\circ f=f\circ s_{PE}\,,\quad t_{PF}\circ f=f\circ t_{PE}\,$,
\item
$\forall\;p,q\in FP(E)\text{ such that } t_{PE}(p)=s_{PE}(q)\colon f(pq)=f(p)f(q)$.
\end{enumerate}
\end{definition}
\noindent If $(f^0,f^1)\colon E\to F$ is a homomorphism of graphs, then $f\colon FP(E)\to FP(F)$ defined by
\begin{gather*}
\forall\; v\in E^0\colon f(v):=f^0(v),\quad
\forall\; e\in E^1\colon f(e):=f^1(e),\\
\forall\; e_1\ldots e_n\in FP(E)\colon f(e_1\ldots e_n):= f^1(e_1)\ldots f^1(e_n)\in FP(F),
\end{gather*}
is a path homomorphism of graphs. The thus induced path homomorphism of graphs preserves the length of
paths.  Vice versa, if a path homomorphism of graphs preserves the length of
paths, then it is induced by a homomorphism of graphs. 
\begin{definition}[\cite{ht24}]
We call a path homomorphism of graphs $f\colon E\to F$ \emph{admissible} when the following three conditions are satisfied:
\begin{enumerate}
\vspace*{-3mm}\item
The map $f\colon FP(E)\to FP(F)$ is injective when restricted to $E^0$ (vertiex-injectivity condition).
\item
If $e$ and $e'$ are edges in $E$, then
$
f(e)\preceq f(e')\;\Longrightarrow\; e=e'
$
(monotonicity condition). 
\item
The map $f\colon FP(E)\to FP(F)$ is regular in the sense of Definition~\ref{regularity} (regularity condition).
\end{enumerate}
\end{definition}

\begin{definition}[\cite{ht24}]\label{regularity}
Let $f\colon FP(E)\to FP(F)$ be apath homomorphism of graphs, and let 
\begin{equation*}
{\rm reg}_0(E):=\{v\in{\rm reg}(E)~|~s_E^{-1}(v)=\{e\}\text{ and } t_E(e)=v\}
\end{equation*} 
denote the set of {\em 0-regular vertices}. We call $f$ \emph{regular} whenever the following conditions hold:
\vspace*{-2mm}\begin{enumerate}
\item
For any $v\in\mathrm{reg}(E)\setminus {\rm reg}_0(E)$, we require that:
\begin{enumerate}
\item
$f$ restricted to $s_E^{-1}(v)$ be injective;
\item
$p\in f(s_E^{-1}(v))$ if and only if
\begin{enumerate}
\item
$\exists\;n\in\mathbb{N}\setminus\{0\}\colon p=e_1\ldots e_n,\, e_i\in E^1 \text{ for all } i$,
\item
$pq\in f(s_E^{-1}(v))\,\Rightarrow\, q=t_{PF}(p)$,
\item
$\forall\; i\in\{1,\ldots,n\},\;e\in s_F^{-1}(s_F(e_i))\;\exists\;r\in \mathrm{FP}(F)\colon
e_1\ldots e_{i-1}er\in f(s_E^{-1}(v))$.
\end{enumerate}
\end{enumerate}
\item 
For any $v\in {\rm reg}_0(E)$, either the above condition holds or $f(s^{-1}_E(v))=f(v)$.
\end{enumerate}
\end{definition}

\begin{theorem}[\cite{ht24}]
Directed graphs and admissible path homomorphisms of graphs form a category, and
we have a covariant functor from this category to the category 
of  C*-algebras and $*$-homomorphisms that assigns graph C*-algebras to graphs and
$*$-homomorphisms to  admissible path homomorphisms as follows. If
$ f : E \to F $ is an admissible path homomorphism of graphs,   then the formula
\begin{align*}
f_* ( S_p )  = S_{ f(p)}, \quad p \in E^0 \cup E^1,
\end{align*}
defines a $*$-homomorphism
$ f_* : C^*(E) \to C^*(F) $.
\end{theorem}

\section{Unital embeddings of Cuntz algebras}
\subsection{Kawamura's embeddings}\label{kawa}

\begin{definition}[\cite{c-j77}]
Let $n\in\mathbb{N}$, $n>1$. The Cuntz algebra
$\mathcal O_n$ is the universal C*-algebra generated by the isometries $S_i$,
$i\in\{1,\ldots,n\}$, satisfying $\sum_{ i =1}^n S_i S_i^* = 1$.
\end{definition}
\noindent
One can easily see that $\mathcal O_n$  is the graph $C^*$-algebra of the graph
\begin{equation}
\label{eq:E_n}
E_n:=\quad
\begin{tikzpicture}[auto,swap]
\tikzstyle{vertex}=[circle,fill=black,minimum size=3pt,inner sep=0pt]
\tikzstyle{edge}=[draw,->]
\tikzstyle{cycle1}=[draw,->,out=130, in=50, loop, distance=40pt]
\tikzstyle{cycle2}=[draw,->,out=130, in=50, loop, distance=70pt]
   
\node[vertex] (0) at (0,0) {};
\node (1) at (0,1.25) {\vdots};

\path (0) edge[cycle1] node {$e_1$} (0);
\path (0) edge[cycle2] node[above] {$e_n$} (0);

\end{tikzpicture}.
\end{equation}
It is equally straightforward to note that, when considering a path homomorphism $h\colon E_m\to E_n$, where
$m-1 = (n-1)k$ for some $k\in\mathbb{N}\setminus\{0\}$, the regularity condition of Definition~\ref{regularity}
dictates the Kawamura formula for the unital embedding  $\mathcal O_m\to\mathcal O_n$.

Indeed, let us  first  label the edges of $E_m$ and $E_n$ by  $\set{e_i}_{ i=1}^m$ and 
$ \{e'_j\}_{ j=1}^n$, and the unique vertices of $E_m$ and $E_n$ by $v$ and  $v'$, respectively.
Of course, $h(v) := v'$. For the edges, 
according to the congruence relation, the first $m-1$ edges of  $E_m$
can be evenly divided into $k$-batches containing  $n-1$ edges.
For the first batch, we set 
\begin{align}
 h (e_j) := e'_j, \, \, j=1, \ldots , n-1. 
\end{align}
For the $l$-th batch,  $l=2, \ldots ,k$, 
\begin{align}
 h (e_{ j + l(n-1)}) := e_n^{l-1} e'_{ j}, \, \, j=1, \ldots , n-1.  
\end{align}
Finally, for the last edge, we set $ h (e_m) := e_n^k$. Now one can readily verify that $h\colon E_m\to E_n$ is an admissible path 
homomorphism and that $h_*\colon\mathcal O_m\to\mathcal O_n$ coincides with the $*$-homomorphism of Kawamura.

\subsection{Matrix algebras over Cuntz algebras as graph C*-algebras} 
We will present two families of graphs whose graph C*-algebras are 
matrix algebras over Cuntz algebras.
Recall first that the matrix algebra  $M_k (\mathbb{C})$ is the  graph  $C^*$-algebra
of the $k$-line graph  
\begin{align}
 \begin{tikzcd}[ampersand replacement=\&]
	{v_1} \& {v_2} \& \cdots \& {v_{k-1}} \& {v_k}
	\arrow[ "{l_1}"', from=1-1, to=1-2]
	\arrow["{l_{k-1}}"', from=1-4, to=1-5]
	\arrow[from=1-2, to=1-3]
	\arrow[from=1-3, to=1-4]
\end{tikzcd}
.
\label{eq:k-line}
\end{align}
Now we can view $M_k ( \mathcal O_m )$ as the universal C*-algebra generated by
\begin{align}
\set{ S_{ l_j } \otimes E_i \;|\; 1 \le j \le k-1 ,\, 1\le i \le m }
\end{align}
subject to appropriate relations. 
On the other hand,  by blowing up the vertex of 
$E_m$ in \eqref{eq:E_n} to the $k$-line graph 
\eqref{eq:k-line}, we obtain:
\begin{equation}
\label{eq:Gmk}
\begin{tikzpicture}[auto,swap,  scale=1.75]
\tikzstyle{vertex}=[circle,fill=black,minimum size=3pt,inner sep=0pt]
\tikzstyle{edge}=[draw,-{Stealth}]
\tikzstyle{cycle1}=[draw,-{Stealth},out=130, in=50, loop, distance=40pt]
\tikzstyle{cycle2}=[draw,-{Stealth},out=100, in=30, loop, distance=40pt]
   
\node[vertex, label=below:$v_1$] (0) at (0,0) {};
\node[vertex, label=below:$v_2$] (1) at (1,0) {};
\node[vertex] (2) at (2,0) {};
\node[vertex] (3) at (3,0) {};
\node[vertex, label=below:$v_k$] (4) at (4,0) {};
\node  at (-1,0) {$ G_{m,k}:=$};

\path (0) edge[edge] node[above, scale=0.7]{} (1) ;
\path (1) edge[edge] node[above, scale=0.7]{} (2);
\path (2) edge[dashed] (3);
\path (3) edge[edge]  node[above, scale=0.7]{} (4);
\path (4) edge[edge, bend left=40 ,red]
node[above, scale=0.7] {$\set{e_1 , \ldots , e_m}$}(0);
\end{tikzpicture}
\end{equation}

By direct verification, we prove that
\begin{proposition}\label{cmat}
The formulas
\begin{align*}
\varphi ( \bar S_{ l_j }) = S_{ l_i } \otimes 1 ,\; 1\leq j\leq k-1,\quad
\varphi ( \bar S_{ e_i }) = S_p \otimes  E_i,\;  1\leq i\leq m,\;p:= l_1 \ldots l_{ k-1}. 
\end{align*}
define a $*$-isomorphism
$ \varphi : C^*(G_{m,k}) \to M_k ( \mathcal O_m )$.
\end{proposition}

Next,
for integers $m,n$ bigger than $1$ and such that $m - 1 = (n-1)k$ for some  $k \ge 1$,
we consider the following graph, where each red arrow represents $n-1$-edges:  
\begin{equation}
\label{eq:Fmn}
\hspace*{4mm}\begin{tikzpicture}[auto,swap, scale=1.75]
\tikzstyle{vertex}=[circle,fill=black,minimum size=3pt,inner sep=0pt]
\tikzstyle{edge}=[draw,-{Stealth}]
\tikzstyle{cycle1}=[draw,-{Stealth},out=130, in=50, loop, distance=40pt]
\tikzstyle{cycle2}=[draw,-{Stealth},out=100, in=30, loop, distance=40pt]
   
\node[vertex, label=below:$v_1$] (0) at (0,0) {};
\node[vertex, label=below:$v_2$] (1) at (1,0) {};
\node[vertex] (2) at (2,0) {};
\node[vertex] (3) at (3,0) {};
\node[vertex, label=below:$v_k$] (4) at (4,0) {};
\node  at (-1,0) {$ F_{m,n}:=$};

\path (0) edge[edge] (1);
\path (1) edge[edge] (2);
\path (2) edge[dashed] (3);
\path (3) edge[edge] (4);
\path (4) edge[edge, bend left=50, red] (3);
\path (4) edge[edge, bend left=50, red] (2);
\path (4) edge[edge, bend left=50, red] (1);
\path (4) edge[edge, bend left=50, red] (0);
\path (4) edge[edge, bend right=40] node[above, scale=0.7] {$e_m$}(0);
\path (4) edge[cycle1, red] node[above, scale=0.5] {$n-1$} (4);
\end{tikzpicture}.
\end{equation}
In more detail, it is obtained by adding $m$ edges to the  $k$-line graph 
\eqref{eq:k-line}, with  $v_k$ as the source, in the following way:
the edge $e_m$ goes to $v_1$, and the remaining $m-1 = k (n -1 )$ edges
are divided evenly  into $k$ batches of $(n -1 )$ with each batch ending in a different vertex. 

Now, it is straghtforward to prove:
\begin{lemma}\label{aux}
Let  $m,n$ be integers bigger than $1$ and such that $m - 1 = (n-1)k$ for some  $k \ge 1$. Then the formulas
\begin{gather*}
f(\bar l_j):= l_j,\;1\leq j\leq k-1,\quad f(\bar e_{i}):=e_i,\, (n-1)(k-1)<i\leq m,\\
 f(\bar e_{i}):=e_i l_1\ldots l_j,\; 1\leq j\leq k-1,\, i=(n-1)(k-1-j) +r,\, 0< r\leq n-1,
 \end{gather*}
define an admissible path homomorphism $f\colon F_{m,n}\to G_{m,k}$ inducing a $*$-isomorphism 
\[
f_*\colon C^*(F_{m,n})\longrightarrow  C^*(G_{m,k}).
\]
\end{lemma}

\subsection{The unital embeddings $ \mathcal O_p \to M_k ( \mathcal O_q )$}
We begin by proving our key lemma which manifests the pivotal role played by the graph $F_{m,n}$.
\begin{lemma}\label{pivot}
Let  $m,n$ be integers bigger than $1$ and such that $m - 1 = (n-1)k$ for some  $k \ge 1$. 
Then the graph $F_{m,n}$ has $k$ vertices and $nk$  edges with $n$ edges ending at each vertex, and
 mapping all vertices to one vertex and $k$ edges each ending at a different vertex to one loop
defines an admissible graph homomorphism $g\colon F_{m,n}\to E_n $ inducing a unital injective $*$-homomorphism 
\[
g^*\colon  C^*(E_n)\longrightarrow  C^*(F_{m,n}).
\]
\end{lemma}

We are now ready to prove the main result:
\begin{theorem}
Let $p,q,k,s\in\mathbb{N}\setminus \{0\}$ with $p,q\geq2$. Assume that these numbers satisfy the congruence
\[
(p-1)k = (q-1)s \;\text{for some}\; s\in\mathbb{N}\setminus\{0\}.
\]
 Then
there exists a finite sequence of admissible graph morphisms inducing injective unital $*$-homomorphisms between graph C*-
algebras whose composition is a unital embedding 
\[
\mathcal O_p\longrightarrow M_k(\mathcal O_q). 
\]
\end{theorem}
\begin{proof}
First, assume that the following three congruences hold
\[
(q-1)r=m-1=(n-1)k,\quad p-1=(n-1)l,\quad l,r\in\mathbb{N}\setminus\{0\}.
\]
Then, by Lemma~\ref{pivot} and Lemma~\ref{aux}, there are admissible graph morphisms
\[
E_n\stackrel{g}{\longleftarrow}F_{m,n}\stackrel{f}{\longrightarrow}G_{m,k}
\]
inducing contravariantly and covariantly, respectively, unital injective $*$-homomorphisms between their graph C*-algebras:
\[
\mathcal O_n=C^*(E_n)\stackrel{g^*}{\longrightarrow}C^*(F_{m,n})\stackrel{f_*}{\longrightarrow}C^*(G_{m,k})
\cong M_k(\mathcal O_m).
\]
Here the last identification comes from Proposition~\ref{cmat}. 

Next, remembering the assumed congruences and using Kawamura's embeddings from Section~\ref{kawa}, 
we obtain the composed unital embedding:
\[
\mathcal O_p\stackrel{h_*}{\longrightarrow}\mathcal O_n\stackrel{f_*\circ g^*}{\longrightarrow}M_k(\mathcal O_m)
\stackrel{e_*}{\longrightarrow}M_k(\mathcal O_q).
\]
Here $E_p\stackrel{h}{\rightarrow}E_n$ and $G_{m,k}\stackrel{e}{\rightarrow}G_{q,k}$ are appropriate path 
homomorphisms of graphs.

Finally, given the congruence $(p-1)k = (q-1)s$, we obtain the above assumed three 
congruences by defining $r:=s$, $l=1$,  $n:=p$, and
$m:=(q-1)s+1$.
\end{proof}

\section*{Acknowledgements}
\noindent
This research is part of the EU Staff Exchange project 101086394 \emph{Operator Algebras That One Can See}.
The project is co-financed by the Polish Ministry of Education and Science under the program PMW (grant agreement 5305/HE/2023/2).
The authors are very grateful to  Jack  Spielberg for inspiration, to Francesco D'Andrea and Mariusz Tobolski for discussions, 
and to S\o ren Eilers for advice. The first author is also
delighted to thank the Arizona State University in Tempe,
where the work on this paper got started, and to SISSA in Trieste, where the  work on this paper was carried out, for great hospitality.

\end{document}